\documentclass{thomas}

\usepackage[latin1]{inputenc}
\usepackage[english]{babel} 
\usepackage[all]{xy}  
\usepackage{color}

\begin{document}

\FirstPageHeading

\ShortArticleName{Affine bundles are affine spaces over modules}

\ArticleName{Affine bundles are affine spaces over modules}

\Author{Thomas LEUTHER}

\AuthorNameForHeading{T.~Leuther}

\Address{Institute of Mathematics, Grande Traverse 12, B-4000 Li\`ege, Belgium}
\Email{thomas.leuther@ulg.ac.be}

\Abstract{
We show that the category of affine bundles over a smooth manifold $M$ is equivalent to the category of affine spaces modelled on projective finitely generated ${\rm C}^\infty(M)$-modules. 
%This equivalence of categories is somehow ``modelled'' on the equivalence of categories between vector bundles and projective finitely generated ${\rm C}^\infty(M)$-modules. 
%We emphasize the similarity between both affine concepts using the language of torsors. 
Using this equivalence of categories, we are able to give an alternate proof of the main result of \cite{llz}, showing that the characterization of vector bundles by means of their Lie algebras of homogeneous differential operators also holds for vector bundles of rank $1$ and over any base manifolds. %(without any topological assumption). 
}

% List of Keywords for your article separated by semicolon.

\Classification{14R10, 55R25} % e.g. 35A30; 81Q05
% For 2010 Mathematics Subject Classification see http://www.ams.org/mathscinet/msc/msc2010.html

%%
\section{Introduction}

Affine bundles appears in differential geometry mainly from the geometry of a frame-independent formulation of classical mechanics (see for instance \cite{ggu, ggu2,ggu3,gur,u,m,mms,bt,b,popescu,popescu2}) and in the study of jet bundles (see \cite{saunders,gm}). 

As affine spaces are modelled on vector spaces, each affine bundle has an underlying vector bundle. This being said, remember (see \cite{swan,nestruev}) that the functor of global sections provides an equivalence of categories between vector bundles over a smooth manifold $M$ and projective finitely generated ${\rm C}^\infty(M)$-modules. 

In this text, we first emphasize the similarities between affine bundles and affine spaces over modules using the language of torsors (see \cite{fga,giraud}). Then we show that the category of affine bundles over a smooth manifold $M$ is equivalent to the category of affine spaces modelled on projective finitely generated ${\rm C}^\infty(M)$-modules. 

Finally, by means of this equivalence of categories and of the concept of \emph{smooth envelope} of an ${\mathbb R}$-algebra (see \cite{nestruev}), we are able to give an alternate proof of the main result of \cite{llz}, showing that the characterization of vector bundles by means of their Lie algebras of homogeneous differential operators (resp. by means of their $ {\mathbb R}$-algebras of fiberwise polynomial functions) also holds for vector bundles of rank $1$ and without any topological assumption about the base manifolds.

\section{Affine spaces}

First, we translate the definitions of affine spaces modelled on vector spaces and affine maps between them into the language of torsors. This allows us to extend in a natural way the definition to \emph{affine spaces over arbitrary modules}.

\subsection{Affine spaces modelled on vector spaces}

Let $V$ be a vector space. Remember that an \emph{affine space} modelled on $V$ is a (nonempty) set $A$ on which $V$ acts freely and transitively, i.e. together with a map
\[
{\rm t} : V \times A \to A : (v,a) \mapsto {\rm t}_v(a)
\]
such that for any two $a_0, a \in A$, there is a unique $v \in V$ such that $a = {\rm t}(v,a_0)$. We often write $a+v$ instead of ${\rm t}_v(a)$ and $a-a_0$ for the unique $v$ ``moving'' $a_0$ to $a$.

Note that the above definition of an affine space makes no use of scalar multiplication in the vector space. What is required is a free and transitive \emph{group} action: an affine space is nothing but a \emph{torsor} under the action of a special kind of group (a vector space).
We are thus dealing with torsors under the action of groups with an additionnal structure: scalar multiplication. Morphisms of affine spaces are morphisms of torsors that are compatible with this richer structure: if $V$ and $V'$ are vector spaces and if $A$ and $A'$ be affine spaces respectively modelled on $V$ and $V'$, a map $f : A \to A'$ is called an \emph{affine map} if there is a linear map $\vec{f} : V \to V'$ such that the induced action of $V$ on $A'$, 
  \[
  V \times A' \to A' : (v,a') \mapsto \vec{f}(v) \cdot a' \; ,
  \]
  turns $f : A \to A'$ into a $V$-morphism of torsors, i.e., $ f(a_0 + v) = f(a_0) + \vec{f}(v)$  for all $a_0 \in A$ and all $v \in V$.

\subsection{Affine spaces modelled on modules}%

In view of what preceeds, nothing prevents us from defining affine spaces modelled on modules over \emph{arbitrary} commutative rings.

\begin{definition}
Let $R$ be a commutative ring and $P$ be an $R$-module. An \emph{$R$-affine space} modelled on $P$ is a $P$-torsor, i.e., a set $A$ together with a free and transitive group action
\[ 
{\rm t} : P \times A \to A : (p,a) \mapsto {\rm t}_p(a) \; .
\]
We also write $a+p$ instead of ${\rm t}_p(a)$ and $a-a_0$ for the unique $p$ ``moving'' $a_0$ to $a$.
\end{definition}

\begin{definition}
Let $R$ be a commutative ring, $P$ and $P'$ be $R$-modules\footnote{Since we consider modules over a commutative ring, we don't distinguish between left and right modules.}. If $A$ and $A'$ are affine spaces respectively modelled on $P$ and $P'$, a map $f : A \to A'$ is called an \emph{$R$-affine map} if there is an $R$-linear map $\vec{f} : P \to P'$ (the \emph{$R$-linear part} of $f$) such that the induced action of $P$ on $A'$, 
  \[
  P \times A' \to A' : (p,a') \mapsto \vec{f}(p) \cdot a' \; ,
  \]
  turns $f : A \to A'$ into a $P$-morphism of torsors. 
\end{definition}

The composition of $R$-affine maps is obviously an $R$-affine map. Also the identity maps ${\rm id}_A : A \to A$ are $R$-affine maps. We denote by ${\rm AS}({\rm Mod\ } R)$ the category of $R$-affine spaces and $R$-affine maps between them. 

\begin{remark}
The category ${\rm AS}({\rm Mod\ } R)$ is \emph{not} a subcategory of the category of torsors because $R$-affine maps are not necessarily equivariant maps. Roughly speaking, equivariance is ``shifted'' by the $R$-linear map $\vec{f}$.
\end{remark}

\section{Vector bundles}

%\subsection{Vector bundles as group objects}%

\begin{definition}
Let $M$ be a connected $m$-dimensional smooth manifold. A fiber bundle $\eta : E \to M$ is a \emph{vector bundle} if
\begin{enumerate}
  \item there is vector space $V$ ((the \emph{typical fiber}) such that for any $x \in M$, the fiber $\eta_x$ is a vector space isomorphic to $V$ ;
  \item the base manifold $M$ is entirely covered with the domains $U$ of fiberwise linear local trivializations $\Phi_U : \eta^{-1}(U)
 \stackrel{\sim}{\longrightarrow} U \times V$:
  \[
  \xymatrix{
  \eta^{-1}(U) \ar[rr]^{\Phi_U} \ar[dr]_{\eta} && U \times V \ar[dl]^{{\rm pr}_1} \\
  & U &
  }
  \]
\end{enumerate}
\end{definition}

Any vector bundle $\eta : E \to M$ admits a canonical global section: the \emph{zero section}. The latter associates to any point $x \in M$ the origin $0_x$ of the vector space $\eta_x$.

\begin{remark}\label{prop:total space reached by global sections}
Any point in the total space of a vector bundle $\eta : E \to M$ can be reached by a \emph{global} section. Indeed, any point in the total space of a fiber bundle can be reached by a local section (see \cite{Husemoller}) while multiplication by a ``bump'' function on $M$ allows use to globalize such a local section.
\end{remark}

\begin{proposition}
Any vector bundle is a group object in the category ${\rm FB}(M)$ of fiber bundles over $M$.
\end{proposition}
\begin{proof}
Define a triple $(m, {\rm inv}, \epsilon)$ as follows:
\[
\begin{array}{rcl}
m & : & E \times_M E \to E : (e,e') \mapsto e + e' \in \eta_x \; ;\\
&&\\
\epsilon & : & M \to E : x \mapsto 0_{x} \in \eta_x \; ;\\
&&\\
{\rm inv} & : & E \to E : e \mapsto - e \in \eta_x \; .\\
\end{array}
\]
These maps are smooth and fiber-preserving. Moreover, they define a group structure on each fiber since they are build fiberwise from vector space structures. It follows that $\eta : E \to M$ is indeed a (particular kind of) group object in ${\rm FB}(M)$.
\end{proof}

% 
% \begin{remark}[Vector bundles are not principal bundles]
% ???
% \end{remark}

%\subsection{The category ${\rm VB}(M)$}

Vector bundles are fiber bundles with a richer structure: fibers are vector spaces. A morphism of vector bundles is not any morphism of fiber bundles: it must be compatible with the vector structure.

\begin{definition}
An \emph{$M$-morphism of vector bundles} from $\eta : E \to M$ to $\eta'
: E' \to M$ is a fiberwise linear $M$-morphism of fiber bundles, i.e. an
$M$-morphism such that all restrictions ${\left.f\right|}_{\eta_x} : \eta_x \to
\eta'_x$ are linear maps.
\end{definition}

Vector bundles over $M$ and $M$-morphisms between them form a category, denoted by ${\rm VB}(M)$. This category is well-known to be equivalent to a certain category of modules: 
\begin{proposition}
The functor of global sections $\Gamma$, from the category ${\rm VB}(M)$ of vector bundles over $M$ to the category ${\rm Mod}_{\text{pfg}}\ {\rm C}^\infty(M)$ of projective finitely generated ${\rm C}^\infty(M)$-modules is an equivalence of categories.
\end{proposition}
\begin{proof}
See  \cite{swan,nestruev}.
\end{proof}

\begin{remark}
It is not useful here to enter into details about the notion of projective finitely generated modules. The interested reader is invited to consult \cite{swan,nestruev}.
\end{remark}

\section{Affine bundles}

%\subsection{Affine bundles as torsors}

\begin{definition}
Let $M$ be a connected $m$-dimensional smooth manifold and $\eta : E \to M$ be a vector bundle. A fiber bundle $\pi : Z \to M$ is an \emph{affine bundle} modelled on $\eta$ if
\begin{enumerate}
  \item there is an affine space $A$ (the \emph{typical fiber}, modelled on the typical fiber $V$ of $\eta$) such that for any $x \in M$, the fiber $\pi_x$ is an affine space modelled on $\eta_x$, isomorphic to $A$ ;
  \item the base manifold $M$ is entirely covered with the domains $U$ of fiberwise affine local trivializations $\Phi_U : \pi^{-1}(U)
 \stackrel{\sim}{\longrightarrow} U \times A$:
  \[
  \xymatrix{
  \pi^{-1}(U) \ar[rr]^{\Phi_U} \ar[dr]_{\pi} && U \times A \ar[dl]^{{\rm pr}_1} \\
  & U &
  }
  \]
\end{enumerate}
\end{definition}

\begin{remark}
Since fibers in an affine bundle are affine spaces, they don't come with a preferred origin. Consequently, no canonical global section shows up. However, since the fibers are diffeomorphic to a euclidean space, it follows from \cite{steenrod} that any affine bundle admits global sections.
\end{remark}

\begin{proposition}
Any affine bundle $\pi : Z \to M$ modelled on $\eta : E \to M$ is a torsor under the action of the group object $\eta$ in the category ${\rm FB}(M)$.
\end{proposition}
\begin{proof}
Fiberwise translations in all fibers define a map
\begin{eqnarray*}
{\rm t} : E \times_M Z & \to     & Z \\
    (e,z)        & \mapsto & z + e \in \pi_{\pi(z)} 
\end{eqnarray*}
This map is smooth. Indeed, there exists local adapted coordinates in which ${\rm t}$ reads
\[
{{\rm t}(x^i;e^\beta;z^\alpha)} \; .
% _{\begin{subarray}{l} 
% i = 1,\ldots,m \\
% \alpha, \beta = 1,\ldots,n
% \end{subarray}}
= (x^i;z^\alpha+e^\alpha)
% _{\begin{subarray}{l} 
% i = 1,\ldots,m \\
% \alpha = 1,\ldots,n
% \end{subarray}}
\]
It makes $\pi : Z \to M$ a torsor under the group object $\eta : E \to M$ in the category ${\rm FB}(M)$ since by construction, the restriction of ${\rm t}$ on fibers is build from free and transitive actions.
\end{proof}

\begin{proposition}
Conversely, any torsor $\pi : Z \to M$ under the action of a vector bundle $\eta : E \to M$ is an affine bundle modelled on $\eta$.
\end{proposition}
\begin{proof}
The action making $\pi : Z \to M$ a torsor under $\eta : E \to M$ defines on each fiber $\pi_x$ the structure of an affine space over the corresponding vector space $\eta_x$. Take $A$ to be any affine space modelled on the typical
fiber $V$ of $\eta$ and build fiberwise affine trivializations from the choice of any global section $s_0 : M \to Z$ of $\pi$ and from fiberwise linear trivializations of $\eta$.
\end{proof}

\begin{definition}
Let $\pi : Z \to M$ and $\pi': Z' \to M$ be affine bundles modelled on $\eta : E \to M$. An \emph{$\eta$-morphism of affine bundles} from $\pi : Z \to M$ to $\pi' : Z' \to M$ is an $\eta$-morphism of torsors.
\end{definition}

\begin{proposition}
A smooth map $f : Z \to Z'$ is an $\eta$-morphism of affine bundles if and only if it is an $M$-morphism of fiber bundles such that all restrictions ${\left.f\right|}_{\pi_x} : \pi_x \to \pi'_x$ are $\eta_x$-affine maps.
\end{proposition}
\begin{proof}
The equivariance of $f$ under the action of the vector bundle translates fiberwise into equivariance of the restrictions $f_x : \pi_x \to \pi'_x$.
\end{proof}

%\subsection{The category ${\rm AB}(M)$}

Affine bundles are fiber bundles with some extra structure: fibers are affine spaces.  A morphism of affine bundles is not any morphism of fiber bundles: it must be compatible with the affine structure.

\begin{definition}
Let $\pi : Z \to M$ and $\pi': Z' \to M$ be affine spaces modelled on $\eta : E \to M$ and $\eta : E' \to M$. An \emph{$M$-morphism of affine bundles} from $\pi : Z \to M$ to $\pi' : Z' \to M$ is fiberwise affine $M$-morphism of fiber bundles, i.e. an $M$-morphism such that all restrictions ${\left.f\right|}_{\pi_x} : \pi_x \to \pi'_x$ are affine maps.
\end{definition}

The following proposition shows that the above definition corresponds in some sense to comparing torsors under the action of two different vector bundles over $M$: morphisms of affine bundles have to be compatible with the structure of the underlying vector bundles.

\begin{proposition}
An $M$-morphism of fiber bundles $f : Z \to Z'$ is an $M$-morphism of affine bundles if and only if there is an $M$-morphism of vector bundles $\vec{f}$ (the \emph{linear part} of $f$) from $\eta : E \to M$ and $\eta' : E' \to M$ such that for any $x \in M$, the restriction ${\left.f\right|}_{\pi_x} : \pi_x  \to \pi'_x$ is an affine map with linear part ${\left.\vec{f}\right|}_{\eta_x}: \eta_x \to \eta'_x$.
\end{proposition}
\begin{proof}
We define $\vec{f}:E \to E'$ on each fiber $\eta_x$ as the linear part of ${\left.f\right|}_{\pi_x} : \pi_x \to \pi'_x$ and we check easily that it is indeed an $M$-morphism of vector bundles.
\end{proof}

Affine bundles over $M$ and $M$-morphisms between them form a category, denoted by ${\rm AB}(M)$. 

\begin{remark}
The category ${\rm AB}(M)$ and affine maps is \emph{not} a subcategory of the category of torsors under the action of vector bundles since morphisms are not necessarily equivariant maps. Here, equivariance is ``shifted'' by an $M$-morphism of vector bundles.
\end{remark}

\section{An equivalence of categories}

Both affine bundles  and affine spaces over projective finitely generated ${\rm C}^\infty(M)$-modules were shown above to be torsors uner the action of particular group objects. Also morphisms in both cases were required to preserve the same kind of ``linear property'' of these group objects. Finally, these group objects live in two {equivalent categories} \cite{husemoller, nestruev}. Let us show that the two affine concepts inherit this {equivalence of categories}.

%\subsection{An equivalence of categories}

\begin{lemma}
Let $M$ be a connected $m$-dimensional smooth manifold and $\pi : Z \to M$ be an affine bundle modelled on a vector bundle $\eta : E \to M$. The space of global sections $\Gamma(\pi)$ is in a canonical way an affine space over the ${\rm C}^\infty(M)$-module $\Gamma(\eta)$.
\end{lemma}
\begin{proof}
Fiberwise translations in all fibers define an action
\[
{\rm t} : \Gamma(\eta) \times \Gamma(\pi) \to \Gamma(\pi) : (u,s) \mapsto s + u 
\]
where the value of the section $s+u$ at a point $x \in M$ is given by
\[
(s+u)(x) := s(x) + u(x) \in \pi_x
\]
This action free and transitive since it is build fiberwise from affine spaces structures.
\end{proof}

\begin{lemma}\label{prop:affine bundle reached by sections}
Any point $z$ in the total space of an affine bundle $\pi : Z \to M$ can be reached by a global section.
\end{lemma}
\begin{proof}
Having fixed a global section $s_0 : M \to Z$, it suffices to find a global section $u$ of the vector bundle $\eta : E \to M$ passing through $z - s_0(\pi(z)) \in \eta_{\pi(z)}$ because then $s_0+u$ does the job. Since Remark \ref{prop:total space reached by global sections} ensures the existence of such a section $u$, the Lemma follows.
\end{proof}

\begin{theorem}\label{theorem: equiv of categories affine}
The functor of global sections $\Gamma$, from the category ${\rm AB}(M)$ of affine bundles over $M$ to the category ${\rm AS}({\rm Mod}_{\text{pfg}}\ {\rm C}^\infty(M))$ of affine spaces over projective finitely generated ${\rm C}^\infty(M)$-modules is an equivalence of categories.
\end{theorem}
\begin{proof}
Remember that a functor is an equivalence of categories if and only if it is essentially surjective
and fully faithful.

First, the functor $\Gamma$ is essentially surjective. Let $A$ be a ${\rm C}^\infty(M)$-affine space modelled on a projective finitely generated module $P$. From the equivalence of categories between vector bundles and projective finitely generated ${\rm C}^\infty(M)$-modules, there is a vector bundle $\eta : E \to M$ and a ${\rm C}^\infty(M)$-isomorphism $T : P
  \stackrel{\sim}{\longrightarrow} \Gamma(\eta)$. The ${\rm C}^\infty(M)$-module $\Gamma(\eta)$ is also a ${\rm C}^\infty(M)$-affine space modelled on itself. Picking any point $a_0 \in A$,
  %(and the zero section in $\Gamma(\eta)$),
 we define an isomorphism of ${\rm C}^\infty(M)$-affine spaces $\tilde{T} : A \stackrel{\sim}{\longrightarrow} \Gamma(\eta)$ by setting $\tilde{T}(a_0 + p) := T(p)$.
  
The functor $\Gamma$ is also faithful. Let $\alpha$ and $\beta$ be $M$-morphisms of affine bundles from $\pi : Z \to M$ to $\pi' : Z' \to M$, such that $\Gamma(\alpha) = \Gamma(\beta)$. This equality means that $\alpha \circ s = \beta \circ s$ for every $s \in \Gamma(\pi)$. It follows that $\alpha = \beta$ since Lemma \ref{prop:affine bundle reached by sections} ensures that any point in the total space of an affine bundle can be reached by a global section.
  
Finally, the functor $\Gamma$ is full. Let $\pi : Z \to M$ and $\pi' : Z' \to M$ be affine bundles modelled on $\eta : E \to M$ and $\eta' : E' \to M$, respectively. Let $T : \Gamma(\pi) \to \Gamma(\pi')$ be a ${\rm C}^\infty(M)$-affine map. The linear part $\vec{T} : \Gamma(\eta) \to \Gamma(\eta')$ is a ${\rm C}^\infty(M)$-homomorphism. From the equivalence of categories between vector bundles and projective finitely generated ${\rm C}^\infty(M)$-modules, there exists an $M$-morphism of vector bundles $f$ from $\eta : E \to M$ to $\eta' : E' \to M$ such that $\vec{T} = \Gamma(f)$. Picking any global section $s_0 \in \Gamma(\pi)$, we define an $M$-morphism of affine bundles $\alpha : Z \to Z'$ by setting $\alpha(z) := T(s_0)(\pi(z)) + f(z - s_0(\pi(z)))$. Obviously, $\Gamma(\alpha) = T$. 
\end{proof}

\section{An algebraic characterization for vector bundles}

\subsection{Fiberwise polynomial functions on vector bundles}

Let $\pi : E_\pi \to M$ be a vector bundle over $M$. The pullback map $\pi^*$ determines an isomorphism of ${\mathbb R}$-algebras ${{\rm C}^\infty(M)} \to {\rm Pol}^0(\pi)$. 
This way, multiplication by elements of ${\rm Pol}^0(\pi)$ endows every ${\rm Pol}^k(\pi)$ $(k \in {\mathbb N})$ with a natural ${{\rm C}^\infty(M)}$-module structure: for all $g \in {{\rm C}^\infty(M)}, f \in {\rm Pol}^k(\pi)$, 
\begin{equation}\label{eq: module structure on A1} 
g f := \pi^{*}(g)f \in {\rm Pol}^k(\pi) \; .
\end{equation}
In particular when $k = 1$, this ${{\rm C}^\infty(M)}$-module is isomorphic to the dual module $\Gamma(\pi)^\vee$ of the module of (global) sections of $\pi$. For a function $f \in {\rm Pol}^1(\pi)$, the corresponding element $\xi_f \in
\Gamma(\pi)^\vee$ is the onesuch that for any $s \in \Gamma(\pi)$ and any $x \in M$, we have 
\begin{equation}\label{eq: iso A1 - dual}
\xi_f(s)(x) = f(s(x)) \; .
\end{equation}

\begin{proposition}[\cite{llz}]\label{prop:iso are filtered}
Let $\pi : E_\pi \to M$ and $\eta : E_\eta \to N$ be two vector bundles. 
Every isomorphism of ${\mathbb R}$-algebras $\Psi: {{\rm Pol}}(E_\pi) \to
{{\rm Pol}}(E_\eta)$ is filtered with respect to the filtrations of ${{\rm Pol}}(E_\pi)$ and
${{\rm Pol}}(E_\eta)$ associated with their gradings.
\end{proposition}

\subsection{Smooth envelopes and lifting of isomorphisms}

Let $A$ be an \textit{${\mathbb R}$-algebra}, i.e. an associative algebra over ${\mathbb R}$ which is commutative and has a unit element $1_A$. Let ${\rm Spec}_{\mathbb R} A$ denote the set of all ${\mathbb R}$-algebra homomorphisms $A \to {\mathbb R}$. 
The set ${\rm Spec}_{\mathbb R} A$ is by definition the ${\mathbb R}$-\textit{spectrum}, while the elements of ${\rm Spec}_{\mathbb R} A$ will sometimes be called \textit{${\mathbb R}$-points} of the algebra $A$.
+
Elements of $A$ can be viewed as functions on ${\rm Spec}_{\mathbb R} A$. More precisely, we can associate a function with each $a \in A$, namely the function
\begin{eqnarray*}
f_a : {\rm Spec}_{\mathbb R} A \to {\mathbb R} : h \mapsto f_a(h) := h(a)
\end{eqnarray*}
This procedure supply thus  any abstract algebra $A$ with a \emph{geometrization} as an algebra of functions on the space ${\rm Spec}_{\mathbb R} A$.

\begin{example}
In the case where $A$ is the algebra of smooth functions on a smooth manifold, the ${\mathbb R}$-spectrum identifies with the manifold itself (see \cite[paragraph 7.2]{nestruev}).
\end{example}

\noindent Consider the subalgebra $A_{\text{geom}} := \{ f_a : a \in A \}$ of the ${\mathbb R}$-algebra ${\mathcal F}(A)$ of all (real-valued) functions on ${\rm Spec}_{\mathbb R} A$. The  map $A \to A_{\text{geom}}, a \mapsto f_a$ is always a surjective ${\mathbb R}$-homomorphism. When it is injective, we say that the algebra $A$ is
\textit{geometric}. A geometric algebra is thus an algebra which is canonically isomorphic to an algebra of functions through the above homomorphism.  

\begin{lemma}
An ${\mathbb R}$-algebra $A$ is geometric if and only if the ideal $I_A := \bigcap_{h \in {\rm Spec}_{\mathbb R} A} \ker h$ is trivial. In particular, for any vector bundle $\pi : E_\pi \to M$, $A = {\rm C}^\infty(E_\pi)$ and $A={\rm Pol}(\pi)$ are geometric.
\end{lemma}
\begin{proof}
The first statement follows from the fact that we have
\begin{eqnarray*}
f_a = 0 & \Leftrightarrow & \forall h \in {\rm Spec}_{\mathbb R} A \mbox{, } f_a(h) = h(a) = 0 \\
%		& \Leftrightarrow & \forall h \in {\rm Spec}_{\mathbb R} A \mbox{, }  = 0 \\
		& \Leftrightarrow & a \in \bigcap_{h \in {\rm Spec}_{\mathbb R} A} \ker h \; .
\end{eqnarray*}
In the case $A = {\rm C}^\infty(E_\pi)$ or $A={\rm Pol}(\pi)$, the ideal $I_A$ is a subset of $\bigcap \limits_{x \in M} \ker {\rm ev}_x = {0}$ and $I_A$ is thus trivial because only the zero function vanishes at all points of $E_\pi$. 
\end{proof}

\begin{definition}[\cite{nestruev}]%[paragraph 3.32]{nestruev}]
A geometric ${\mathbb R}$-algebra $A$ (viewed as an algebra of functions) is \emph{${\rm C}^\infty$-closed} if for any for any finite collection of its elements $f_1,...,f_k$ and any function $g \in {\rm C}^\infty({\mathbb R}^k)$, there exists an element $f \in A$ such that 
\[
g(f_1(x),...,f_k(x)) = f(x) \quad \mbox{for all $x \in {\rm Spec}_{\mathbb R} A$} \; .
\] 
\end{definition}

\begin{example}
For any vector bundle $\pi : E_\pi \to M$, the algebra $A = {\rm C}^\infty(E_\pi)$ is ${\rm C}^\infty$-closed. Indeed, for any finite collection of elements $f_1,...,f_k \in {\rm C}^\infty(E_\pi)$ and any function $g \in {\rm C}^\infty({\mathbb R}^k)$, the composition of the smooth maps $(f_1,...,f_k): E_\pi \to \mathbb{R}^k$ and $g : {\mathbb R}^k \to {\mathbb R}$ belongs to ${\rm C}^\infty(E_\pi)$.
\end{example}

\begin{definition}[{\cite[paragraph 3.32]{nestruev}}]
Let $A$ be a geometric ${\mathbb R}$- algebra. A ${\rm C}^\infty$-closed ${\mathbb R}$-algebra $\bar{A}$ together with a homomorphism $i : A \to \bar{A}$ is a \emph{smooth envelope} of $A$ if for any homomorphism $\alpha : A \to B$ from $A$ into a ${\rm C}^\infty$-closed ${\mathbb R}$-algebra $B$, there exists a unique homomorphism $\bar{\alpha} : \bar{A} \to B$ extending $\alpha$. In other words, under the above assumptions, the diagram 
\[
\xymatrix{
A \ar[rr]^\alpha \ar[d]_i && B \\
\bar{A} \ar@{.>}[urr]_{\bar{\alpha}} && 
}
\]
can always be uniquely completed (by the dotted arrow $\bar{\alpha}$) to a commutative one.
\end{definition}

\begin{proposition}\label{proposition: smooth envelope Pol}
For any vector bundle $\pi : E_\pi \to M$, the algebra ${\rm C}^\infty(E_\pi)$ (together with the inclusion $i_\pi:{\rm Pol}(\pi) \to {\rm C}^\infty(E_\pi)$) is a smooth envelope of ${\rm Pol}(\pi)$.
\end{proposition}
\begin{proof}
See \cite[paragraph 11.58]{nestruev}.
\end{proof}

\begin{proposition}\label{proposition: extending from Pol}
Let $\pi : E_\pi \to M$ and $\eta : E_\eta \to M$ be two vector bundles over $M$. For any ${\mathbb R}$-isomorphism $\Psi: {{\rm Pol}}(E_\pi) \to {{\rm Pol}}(E_\eta)$, there is a unique ${\mathbb R}$-homomorphism $\bar{\Psi} : {\rm C}^\infty(E_\pi) \to {\rm C}^\infty(E_\eta)$ extending $\Psi$. In other words, under the above assumptions, the diagram 
\[
\xymatrix{
{\rm Pol}(\pi) \ar[rr]^\Psi \ar[d]_{i_\pi} && {\rm Pol}(\eta) \ar[d]^{i_\eta} \\
{\rm C}^\infty(E_\pi) \ar@{.>}[rr]^{\bar{\Psi}} && {\rm C}^\infty(E_\eta) 
}
\]
can always be uniquely completed (by the dotted arrow $\bar{\Psi}$) to a commutative one. Moreover, $\bar{\Psi}$ is an ${\mathbb R}$-isomorphism.
\end{proposition} 
\begin{proof}
The existence and uniqueness of $\bar{\Psi}$ is a direct consequence of Proposition \ref{proposition: smooth envelope Pol}: use the universal property of the smooth envelope of ${\rm Pol}(\pi)$ with $\alpha = i_\eta \circ \Psi : {\rm Pol}(\pi) \to {\rm C}^\infty(E_\eta) $.  

To show that $\bar{\Psi}$ is an isomorphism, we shall find an inverse homomorphism. Let $\tilde{\Psi}$ be the homomorphism ${\rm C}^\infty(E_\eta) \to {\rm C}^\infty(E_\pi)$ obtained from the universal property of ${\rm Pol}(\eta)$ with $\alpha = i_\pi \circ \Psi^{-1} : {\rm Pol}(\eta) \to {\rm C}^\infty(E_\pi)$. To show that  $\tilde{\Psi} \circ \bar{\Psi} = {\rm id}_{{\rm C}^\infty(E_\pi)}$, it suffices to note that there is also a unique extension of the identity map ${\rm id}_{{\rm Pol}(\pi)} : {\rm Pol}(\pi) \to {\rm Pol}(\pi)$ to an ${\mathbb R}$-homomorphism $\bar{\rm id} : {\rm C}^\infty(E_\pi) \to {\rm C}^\infty(E_\pi)$: since ${\rm id}_{{\rm C}^\infty(E_\pi)}$ and $\tilde{\Psi} \circ \bar{\Psi}$ are two such extensions, they must coïncide. Starting from ${\rm id}_{{\rm Pol}(\pi)}$, we show the same way that $\bar{\Psi} \circ \tilde{\Psi} = {\rm id}_{{\rm C}^\infty(E_\eta)}$, hence $\tilde{\Psi} = \bar{\Psi}^{-1}$.
\end{proof}

In view of Milnor's classical theorem, any ${\mathbb R}$-isomorphism ${\rm C}^\infty(E_\pi) \to {\rm C}^\infty(E_\eta)$ is the pullback map of a diffeomorphism $E_\eta \to E_\pi$, hence the following immediate corollary.

\begin{corollary}
If two vector bundles have their algebras of fiberwise polynomial functions isomorphic (as $\mathbb{R}$-algebras), then their total spaces are diffeomorphic.
\end{corollary}

\subsection{An algebraic characterization}

% \begin{lemma}\label{lem:module hull for vector bundles}
% Let $A$ be an $\mathbb{R}$-algebra and $P$ be an $A$-module. 
% %for which the canonical map $P \mapsto P^{\vee\vee}$ is an isomorphism. 
% The map
% \begin{eqnarray*}
% \Aff_{A}(P,A)^\vee & \to & A \times P^{\vee\vee} \\
% T & \mapsto & (T(1_A),{\left.T\right|}_{P^\vee})
% \end{eqnarray*}
% is an isomorphism of $A$-modules.
% \end{lemma}
% \begin{proof}
% Any element $f \in \Aff_{A}(P,A)$ writes uniquely as the sum of a constant and% an $A$-linear part, i.e. $f = f_0 + \vec{f}$ with $f_0 \in A$ and $\vec{f} \in% P^\vee$. By $A$-linearity, 
% \[
% T(f) = f_0.T(1_A)+{\left.T\right|}_{P^\vee}(\vec{f})
% \]
% and the conclusion follows easily.
% \end{proof}
 
\begin{theorem}\label{theorem: algebraic characterization}
Let $\pi : E_\pi \to M$ and $\eta : E_\eta \to M$ be two vector bundles over $M$. A map $\Psi: {{\rm Pol}}(\pi) \to {{\rm Pol}}(\eta)$ satisfying  
\begin{equation}\label{eq: property restriction to A0}
\Psi \circ \pi^* = \eta^*
\end{equation}
is an isomorphism of ${\mathbb R}$-algebras if and only if there exists an isomorphism of \emph{affine} bundles $\alpha: \pi \to \eta$ over ${\rm id}_M$ such that 
\begin{equation}\label{eq: relation Psi - alpha}
\Psi = {\left.\alpha^{-1*}\right|}_{{{\rm Pol}}(\pi)}
\end{equation}
In this case, $\Psi$ is filtered and the induced graded isomorphism
$\Psi^{{\rm gr}} : {\rm Pol}(\pi) \to {\rm Pol}(\eta)$ is related to $\alpha$ via the
formula
\begin{equation}\label{eq: relation Psi - alpha GR}
\Psi^{{\rm gr}} = {\left.\vec{\alpha}^{-1*}\right|}_{{{\rm Pol}}(\pi)}
\end{equation}
where the linear part $\vec{\alpha} : \pi \to \eta$ of $\alpha$ is an
isomorphism of \emph{vector} bundles over $M$.
\end{theorem}
\begin{proof}
From lemma \ref{prop:iso are filtered}, we know that $\Psi$ is necessarily
filtered. It follows that its restriction $\Psi^{\leqslant 1} : {{\rm Pol}}^{\leqslant 1}(\pi) \to {{\rm Pol}}^{\leqslant 1}(\eta)$ is a bijection, while hypothesis (\ref{eq: property restriction to A0}) ensures that this map is ${{\rm C}^\infty(M)}$-linear: for any $g \in {{\rm C}^\infty(M)}, f \in {{\rm Pol}}^{\leqslant 1}(\pi)$, we have
\[
\Psi^{\leqslant 1}(g.f) = \Psi(\pi^*(g))\Psi(f) = \eta^*(g) \Psi(f) = g.\Psi^{\leqslant 1}(f)
\]
Now, the restriction ${\left.\Psi\right|}_{{\rm Pol}^1(\pi)}$ induces a
${{\rm C}^\infty(M)}$-homomorphism
\[
\Gamma(\pi)^{\vee} \cong {\rm Pol}^1(\pi) \stackrel{\Psi}{\hookrightarrow} {{\rm Pol}}^{\leqslant 1}(\eta) 
\]
whose dual map yields a ${{\rm C}^\infty(M)}$-homorphism ${{\rm Pol}}^{\leqslant 1}(\eta)^{\vee}  \to \Gamma(\pi)$ since the module $\Gamma(\pi)$ coincide with its bidual.(\footnote{Indeed, the dual of a module of sections can be seen as the module of sections of the dual bundle, where the duality is that of finite-dimensional vector spaces.}) Composing with the ${\rm C}^\infty(M)$-affine map(\footnote{This affine map appears in the study of \emph{vector hulls} of affine spaces and affine bundles \cite{MR2415174,ggu}.})
\[
{\rm i}_{\Gamma(\eta)} : \Gamma(\eta) \to {\rm Pol}^{\leqslant 1}(\eta)^{\vee} : t \mapsto (t^* : f \mapsto f \circ t) \; ,
\]
one gets a ${\rm C}^\infty(M)$-affine map $T : \Gamma(\eta) \to \Gamma(\pi)$.

In view of Theorem \ref{theorem: equiv of categories affine}, there is a unique morphism of affine bundles $\beta : \eta \to \pi$ over ${\rm id}_M$ for which ${{\rm C}^\infty(M)}$-affine morphism $T : \Gamma(\eta) \to \Gamma(\pi)$ reads $T = \Gamma(\beta)$, i.e. $T(t) = \beta \circ t$ for all $t \in \Gamma(\eta)$. Let us show that we have $\Psi = \left.\beta^*\right|_{{\rm Pol}(\pi)}$ and $\Psi^{\rm gr} = \left.\vec{\beta}^*\right|_{{\rm Pol}(\pi)}$.

Since every point of the total space $E_\eta$ can be reached by a global section of $\eta$, it is enough to show $f \circ \beta(t(x)) = \Psi(f)(t(x))$ and $f \circ \vec\beta(t(x)) =  \Psi^{{\rm gr}}(f)(t(x))$ for all $f \in {\rm Pol}^1(\pi), t \in \Gamma(\eta), x \in M$. On the one hand, we compute successively
\begin{eqnarray*}
f \circ \beta(t(x))
& = & \xi_f(\beta \circ t)(x) \\
& = & \xi_f(T(t))(x) \\
%& = &  \xi_f(s_{F^\vee \circ i_{\Gamma(\eta)}(t)})(x) \\
& = &  {\rm i}_{\Gamma(\eta)}(t) (F(\xi_f))(x) \\
& = &  F(\xi_f) \circ t (x) \\
& = & \Psi(f)(t(x)) \; .
\end{eqnarray*}
On the other hand, we have also 
\begin{eqnarray*}
f \circ \vec\beta(t(x))
& = & \xi_f(\vec{\beta} \circ t)(x) \\
& = & \xi_f(\vec{T}(t))(x) \\
%& = &  \xi_f(s_{F^\vee \circ \overrigharrow{i_{\Gamma(\eta)}}(t)})(x) \\
& = &  \overrightarrow{{\rm i}_{\Gamma(\eta)}(t)} (F(\xi_f))(x) \\
& = & {\rm pr}^1 (F(\xi_f)) \circ t (x) \\
& = & \Psi^{{\rm gr}}(f)(t(x)) \; .
\end{eqnarray*}

Finally, since ${\rm C}^\infty(E_\pi)$ (resp. ${\rm C}^\infty(E_\eta)$) is the smooth envelope of ${\rm Pol}(\pi)$ (resp. ${\rm Pol}(\eta)$), Proposition \ref{proposition: extending from Pol} ensures that $\Psi : {\rm Pol}(\pi) \to {\rm Pol}(\eta)$ extends in a unique way to an ${\mathbb R}$-homorphism ${\rm C}^\infty(E_\pi) \to {\rm C}^\infty(E_\eta)$, this unique extension being an isomorphism because $\Psi$ is. Since $\beta^* : f \mapsto f \circ \beta$ is such an extension, it is an isomorphism and $\beta$ is thus a diffeomorphism. Setting $\alpha := \beta^{-1}$, we get the announced isomorphism of affine bundles.
\end{proof}

\begin{corollary}\label{corollary: iso from Pol}
Two vector bundles $\pi : E_\pi \to M$ and $\eta : E_\eta \to N$ are isomorphic if and only if their algebras of fiberwise polynomial functions are isomorphic (as $\mathbb{R}$-algebras).
\end{corollary}
\begin{proof}
Let $\Psi: {{\rm Pol}}(\pi) \to {{\rm Pol}}(\eta)$ be an isomorphism of ${\mathbb R}$-algebras. If $M=N$ and if (\ref{eq: property restriction to A0}) holds, the result is an immediate consequence of Theorem \ref{theorem: algebraic characterization}.

In the general case, it follows from \cite[Lemma1]{llz} that $\Psi$ induces a diffeomorphism $g : M \to N$. The pullback bundle $g^*\eta : E_{g^*\eta} \to M$ comes with an isomorphism of vector bundles $\tilde{g} : g^*\eta \to \eta$ while the map $\tilde{g}^* \circ \Psi : {\rm Pol}(\pi) \to {\rm Pol}(g^*\eta)$ is an ${\mathbb R}$-isomorphism satisfying (\ref{eq: property restriction to A0}). We can thus invoke Theorem \ref{theorem: algebraic characterization} to show that the bundles $\pi$ and $g^*\eta$ are isomorphic too, hence the conclusion. 
\end{proof}

\begin{corollary}
Two vector bundles are isomorphic if and only if their Lie algebras of homogeneous differential operators are isomorphic (as Lie algebras).
\end{corollary}
\begin{proof}
It follows immediately from \cite[Proposition 4]{llz} and Corollary \ref{corollary: iso from Pol}.
\end{proof}

%%%%%%
\appendix %
%%%%%%

%%
\section{Torsors}\label{section:torsors in category}

\subsection{Torsors under the action of groups}

\begin{definition}
Let $G$ be a group. A \emph{torsor} under the action of $G$ is a (nonempty) set $X$ on which $G$ acts freely and transitively, i.e. together with a map
\[
\gamma : G \times X \to X : (g,x) \mapsto g \cdot x
\]
such that for any two $x, y \in X$, there is a unique $g \in G$ such that $y = g \cdot x$. 
\bigskip

Asking for the action to be free and transitive amounts to require that the map
\[
\gamma \times {\rm id}_X : G \times X \to X \times X : (g,x) \mapsto (g \cdot x, x)
\]
is a bijection (i.e. an isomorphism in the category of sets).
\end{definition}

Torsors being spaces endowed with a group action, there are natural candidates for morphisms~: equivariant maps.

\begin{definition}
Let $X$ and $X'$ be two torsors under the action of $G$. A map $f : X \to X'$ is a \emph{$G$-morphism of torsors} if it is $G$-equivariant, i.e. for any $x \in X$ and any $g \in g$,
\[
f(g \cdot x) = g \cdot f(x) \; .
\]
In other words, we ask for the following diagram to be commutative :
\[
\xymatrix{
G \times X \ar[rr]^{\gamma} \ar[d]_{{\rm id}_G \times f} && X \ar[d]^{f} \\
G \times X \ar[rr]_{\gamma'}                        && X' }
\]
\end{definition}

More generally, we can define morphisms of torsors under the actions of two {different} groups. 

\begin{definition}
Let $(X, \gamma)$ and $(X', \gamma')$ be torsors under the actions of $G$ and $G'$, respectively. A map $f : X \to X'$ is a \emph{morphism of torsors} if there is a group homomorphism $\vec{f} : G \to G'$ such that the induced action of $G$ on $X'$, 
  \[
  G \times X' \to X' : (g,x') \mapsto \vec{f}(g) \cdot x'
  \]
  turns $f : X \to X'$ into a $G$-morphism of torsors. 
\end{definition}

\subsection{Group objects in a category}\label{subsection: group objects}%

Let ${\mathcal C}$ be a category with finite products (i.e. ${\mathcal C}$ has a terminal object ${\bf 1}$ and any two objects of ${\mathcal C}$ have a product). A \emph{group object} in ${\mathcal C}$ is an object $G$ in ${\mathcal C}$
together with three morphisms:--
\begin{itemize}
  \item the \emph{group multiplication} $m : G \times G \to G$ ;
 
  \item the \emph{inclusion of the unit element} $\epsilon : {\bf 1} \to G$ ;
  
  \item the \emph{inversion operation} ${\rm inv}: G \to G$.
\end{itemize}
These morphisms must fulfill the following requirements:
\begin{itemize}
  \item mutliplication is associative, i.e. $m \circ (m \times {\rm id}_G) = m \circ ({\rm id}_G \times m)$ as morphisms $G \times G \times G \to G$(\footnote{We identify $G \times (G \times G)$ with $(G \times G) \times G$ in a canonical manner.});
  
  \item inclusion of the unit is a two-sided unit of $m$, i.e. $m \circ ({\rm id}_G \times \epsilon) = {\rm pr}_1$, where ${\rm pr}_1 : G \times {\bf 1} \to G$ is the canonical projection, and $m \circ (\epsilon \times {\rm id}_G) = {\rm pr}_2$, where ${\rm pr}_2 : {\bf 1} \times G \to G$ is the canonical projection;
  
  \item inversion operation is a two-sided inverse for $m$, i.e. if $d : G \to G
  \times G$ is the diagonal map, and $\epsilon_G : G \to G$ is the composition
  of the unique morphism $G \to {\bf 1}$ (also called the counit) with $\epsilon$,
  then $m \circ ({\rm id}_G \times {\rm inv}) \circ d = \epsilon_G$ and $m \circ ({\rm inv} \times {\rm id}_G) \circ d = \epsilon_G$.
\end{itemize}

% An equivalent way to define group objects is to say $G$ is a group object in a
% category ${\mathcal C}$ if for every object $X$ in ${\mathcal C}$, there is a
% group structure on the set ${\rm Hom}(X, G)$ such that the association of $X$ to
% ${\rm Hom}(X, G)$ is a contravariant functor (from ${\mathcal C}$ to the category
% of groups).

\subsection{Torsors under the action of group objects}%

Let $G$ be a group object in ${\mathcal C}$. Let $X$ be an object in ${\mathcal C}$. A \emph{group action} of $G$ on $X$ is a morphism $\gamma : G \times X \to X $ in ${\mathcal C}$ such that $\gamma \circ (\epsilon \times {\rm id}_X) = {\rm id}_X$ and $\gamma \circ (m \times {\rm id}_X) =  \gamma \circ ({\rm id}_G \times \gamma)$ as morphisms $G \times G \times X \to G$.(\footnote{Again we identify $G \times (G \times X)$ with $(G \times G) \times X$ in a canonical manner})

\begin{definition}
A \emph{torsor under the action of $G$} is a nonempty object $X$ on which $G$ acts freely and transitively, i.e. the induced map $\gamma \times {\rm id}_X : G \times X \to X \times X$ is an isomorphism in ${\mathcal C}$.
\end{definition}

\begin{definition}
Let $X$ and $X'$ be torsors under the action of a group object $G$. A morphism $f \in {\rm Hom}_{\mathcal C}(X,X')$ is a \emph{$G$-morphism of torsors} if $\gamma' \circ ({\rm id}_G \times f) = f \circ \gamma$ as morphisms $G \times X \to X'$, i.e., if the following diagram commutes:
\[
\xymatrix{
G \times X \ar[rr]^{\gamma} \ar[d]_{{\rm id}_G \times f}  && X \ar[d]^{f} \\
G \times X' \ar[rr]_{\gamma'}                        && X' }
\]
\end{definition}

\section{Fiber bundles}

Affine spaces are particular torsors in the category of sets. Affine bundles will be particular torsors in the category of fiber bundles. Let us recall some basic notion about this category.

\subsection{The category of fiber bundles over a (fixed) manifold}%

Let $M$ be a connected $m$-dimensional smooth manifold. Remember that a \emph{fiber bundle} over $M$ is a surjective smooth submersion $\pi : E \to M$ such that 
\begin{enumerate}
  \item there is a smooth manifold $F$ (the \emph{typical fiber}) such
  that for any $x \in M$, $\pi_x := \pi^{-1}(x)$ is diffeomorphic to $F$ ;
  \item any point $x \in M$ has an open neighboorhood $U$ such that there is a
  diffeomorphism $\Phi_U : \pi^{-1}(U) \to U \times F$ making commutative the
  diagram below:
  \[
  \xymatrix{
  \pi^{-1}(U) \ar[rr]^{\Phi_U} \ar[dr]_{\pi} && U \times F \ar[dl]^{{\rm pr}_1} \\
  & U &
  }
  \]
\end{enumerate}

%\begin{remark}
%Every point in the total space $E$ can be reached by a local section (see \cite{Husemoller ?}).
%\end{remark}

An \emph{$M$-morphism of fiber bundles} from $\pi : E \to M$ to $\pi' : E' \to M$ is a smooth map $f : E \to E'$ that preserves fibers, i.e., $f(\pi_x) \subset \pi'_{f(x)}$for all $x \in M$. 
%In other words, the following diagram commutes:
%\[
%\xymatrix{
%E \ar[rr]^{f} \ar[dr]_{\pi} && E' \ar[dl]^{\pi'} \\
%& M &
%}
%\]
Fiber bundles over $M$ and $M$-morphisms between them form a category, denoted by ${\rm FB}(M)$.

%%\subsection{Products and terminal object in ${\rm FB}(M)$}
%
%The \emph{fibered product} of two fiber bundles $\pi : E \to M$ and $\pi': E' \to M$ is the fiber bundle whose total space is
%\[
%E \times_M E' := \left\{ (e,e') \in E \times E' : \pi(e) = \pi'(e') \right\} 
%\] 
%and whose projection onto $M$ reads $(e,e') \mapsto \pi(e)$. The fibered product of fiber bundles over $M$ comes naturally with two additionnal projections $\pi_1$ and $\pi_2$, on $E$ and $E'$ respectively. 
%
%This fibered product is a \emph{product object} in the sense of category theory: for any fiber bundle $\pi'' : E'' \to M$ and any
%pair of morphisms $f_1 : E'' \to E$, $f_2 : E'' \to E'$, there is a unique $M$-morphism $f : E'' \to E \times_M E'$ making commutative the following two diagrams:
%\[
%\xymatrix{
%E'' \ar[drr]_{f_1} \ar[rr]^{f} && E \times_M E' \ar[d]^{\pi_1} 
%&& 
%E'' \ar[drr]_{f_2} \ar[rr]^{f} && E \times_M E' \ar[d]^{\pi_2} 
%\\
%&& E 
%&& && E'
%}
%\]
%Indeed, this $f$ must satisfy $f(e'') = (f_1(e''),f_2(e''))$ for any $e'' \in E''$, which can be taken as a definition.
%This being said, the simplest fiber bundle over $M$, i.e., ${\rm id}_M : M \to M$ is a terminal object in the sense of category theory: there is precisely one $M$-morphism from any fiber bundle $\pi : E \to M$ to ${\rm id}_M$, namely $\pi$ itself.

\subsection{Group objects and torsors in ${\rm FB}(M)$}%

%In view of to appendix \ref{section:torsors in category},  
A \emph{group object} in ${\rm FB}(M)$ is nothing but a fiber bundle $\eta : E \to M$ endowed with 
\begin{itemize}
  \item a fiber-preserving smooth map $m : E \times_M E \to E$, 
  \item a fiber-preserving smooth map ${\rm inv} : E \to E$,
  \item a global section $\epsilon : M \to E$,
\end{itemize}
so that on each fiber $\eta_x$ ($x \in M$), the induced triple $(m_x, {\rm inv}_x, \epsilon_x)$ defines a group structure.(\footnote{Note that the {fibered product} $\times_M$ of fiber bundles over $M$ defines \emph{product objects} while the trivial bundle ${\rm id}_M : M \to M$ is a \emph{terminal object} in the category ${\rm FB}(M)$, so that we are in the conditions of Paragraph \ref{subsection: group objects}.})

% \begin{remark}
% The triple $(m, {\rm inv}, e)$ induces a factorization of the functor of points
% associated to $\eta$ (i.e. $Hom(-,\eta) : {\rm FB}(M) \to \Ens$) through the
% forgetful functor $\omega : \Grp \to \Ens$. In other words, the functor
% ${\rm Hom}(-,\eta)$ can be seen as valued in the category of groups.
% \end{remark}

\begin{remark}
Note that group objects in ${\rm FB}(M)$ are \emph{not} Lie groups since only elements in the same fiber can be multiplied with each other.
\end{remark}

% Objects in ${\rm FB}(M)$ have a canonical underlying set structure (the total space of the fiber bundle). 
A \emph{group action} of a group object $\eta : E \to M$ on an arbitrary fiber bundle $\pi : Z \to M$ is then an $M$-morphism
\[
\gamma : E \times_M Z \to Z : (e,z) \mapsto e \cdot z
\]
whose restriction to any fiber defines a group action. Finally, a fiber bundle $\pi : Z \to M$ is a \emph{torsor} under $\eta : E \to M$ if the induced map
\[
\gamma \times {\rm id}_Z : E \times_M Z \to Z \times_M Z
\]
is an $M$-isomorphism of fiber bundles (at the level of fibers, this corresponds to asking for all the actions to be free and transitive). 

\begin{example}[Principal bundles]
Let $G$ be a Lie group. A \emph{principal $G$-bundle} is a fiber bundle $\pi : P \to M$ with typical fiber $G$, together with a smooth right action $P \times G \to P$ that is fiber-preserving, free and transitive. 
Nontrivial principal $G$-bundles are \emph{not} group objects in ${\rm FB}(M)$, but rather torsors under the (right) action of the trivial fiber bundle ${\rm pr}_1 : M \times G \to M$. 
In particular, any fiber of a $G$-principal bundle is a torsor (in the category of smooth manifolds) under the action of the Lie group $G$.
\end{example}

%\subsection{Morphisms of fiber bundles over different base manifolds}%
%
%%In order to recover the general definition for morphisms of fiber bundles as pairs of smooth maps, we need to wonder how to relate fiber bundles over different bases.
%
%\begin{definition}
%Let $\pi : E \to M$ and $\pi' : E' \to M'$ be two fiber bundles. A map $f : E
%\to E'$ is a \emph{morphism of fiber bundles} if 
%\begin{enumerate}
%  \item there is a smooth map $\overline{f} : E \to E'$ ;
%  
%  \item the induced bundle (or pull-back bundle) of $\pi'$ by $\overline{f}$ being defined by
%  \begin{eqnarray*}
%  {\rm pr}_1 : \overline{f}^*E := \{ (x,e') : \overline{f}(x) = \pi'(e')) \} 
%  & \to & M' \\ 
%  (x,e')   & \mapsto & x \; ,
%  \end{eqnarray*} 
%  the map $(\pi,f) : E \to \overline{f}^*E$ into an $M$-morphism of fiber bundles.
%\end{enumerate}
%In other words, morphisms of fiber bundles correspond to commutive diagrams
%\[
%\xymatrix{
%E \ar@/^2pc/[rrrr]^{f} \ar[rr]^{(\pi,f)} \ar[d]_{\pi} 
%&& \overline{f}^*E \ar[d]^{{\rm pr}_1} \ar[rr]^{{\rm pr}_2}
%&& E' \ar[d]^{\pi'} \\ 
%M \ar@/_2pc/[rrrr]_{\overline{f}} \ar[rr]_{{\rm id}_M}
%&& M  \ar[rr]_{\overline{f}}
%&& M' 
%}
%\]
%and we often say that $f$ is a morphism above $\overline{f}$.
%\end{definition}

%\nocite{*}

\bibliographystyle{abbrv}
\bibliography{biblio}

\end{document}